\documentclass[11pt]{amsart}

\usepackage{amsmath,amscd}
\usepackage{amssymb}
\usepackage{amsthm}
\usepackage{mathrsfs}
\usepackage{hyperref}

\usepackage{calc}
{\begin{list}{\arabic{enumi}.}{\usecounter{enumi}%
			\setlength{\labelsep}{0.5em}%
			\settowidth{\labelwidth}{(\arabic{enumi})}%
			\setlength{\leftmargin}{\labelwidth+\labelsep}}}%
	{\end{list}}

\newtheorem{Theorem}{Theorem}[section]

\newtheorem{Lemma}[Theorem]{Lemma}

\newtheorem{Proposition}[Theorem]{Proposition}

\theoremstyle{definition}

\newcommand {\R} {\mathbb{R}}

\newtheorem{Remark}[Theorem]{Remark}

\numberwithin{equation}{section}

\title[nonlocal inverse problems]{The Calder\'{o}n problem for nonlocal operators}
\author[Ghosh, Uhlmann]{Tuhin Ghosh$^{*}$, Gunther Uhlmann$^{**}$.}
\address {$*$\quad Department of Mathematics, Universt\"{a}t Bielefeld.
	\newline
	\indent\: 
	E-mail:{\tt\  tghosh@math.uni-bielefeld.de}}
\address{$**$\quad Institute for Advanced Study, Hong Kong University of Science and Technology, and Department of Mathematics, University of Washington.
	\newline
	\indent\: \
	E-mail:{\tt \  gunther@math.washington.edu }}
\date{\today}

\begin{document}
	\begin{abstract}
		We study the inverse problem of determining the coefficients of the fractional power of a general second order elliptic operator given in the exterior of an open subset of the Euclidean space. We show the problem can be reduced into determining the coefficients from the boundary Cauchy data of the elliptic operator on the open set, the Calder\'{o}n problem. As a corollary we establish several new results for nonlocal inverse problems by using the corresponding results for the local inverse problems. In particular the isotropic nonlocal Calder\'{o}n problem can be resolved completely, assuming some regularity assumptions on the coefficients,  and the anisotropic  Calder\'{o}n problem modulo an isometry which is the identity at the boundary for real-analytic anisotropic conductivities in dimension greater than two and  bounded and measurable anisotropic conductivities in two dimensions.
	\end{abstract}
	\maketitle
	\section{Introduction and Statement of the Results}
	In this article, we consider a nonlocal analogue of the classical Calder\'{o}n problem, introduced in \cite{Calderon1980}. The classical Calder\'{o}n problem, also known as electrical impedance tomography (EIT), uses voltage and current measurements at the boundary to determine the conductivity in the interior. 
	
	In mathematical terms, we consider the conductivity equation, assuming no sinks or sources of current, in a smooth bounded domain $\Omega\subseteq\mathbb{R}^n$
	$$-div(\gamma\nabla)u=0\quad\mbox{in }\Omega.$$ 
	The question is  whether making
	votages and current measurements at the boundary i.e. the 
	\textit{boundary Cauchy data} $\mathcal{C}_\gamma =(u|_{\partial\Omega},\gamma\partial_\nu u|_{\partial\Omega})$
	determines the  conductivity $\gamma$ in $\Omega$.
	
	 If one assumes the conductivity is isotropic and sufficiently smooth then do recover it in the interior from the voltage current measurements at the boundary. See section 4 for the precise regularity assumptions. If the conductivity is anisotropic, then the unique recovery fails, there is a gauge invariance. We refer readers to the survey article \cite{UG} for more details and Section \ref{sec4} of this paper.
	
	The fractional Schr\"{o}dinger equation exterior Cauchy data problem was  studied in \cite{GSU}
	\[\left((-\Delta)^a + q\right)v = 0\quad\mbox{in } \Omega, \quad supp\, v\subseteq \overline{\Omega}\cup \overline{W},\quad \overline{\Omega}\cap \overline{W}=\emptyset.\]
	Here $W$ denotes an open set. The fractional Laplacian is defined by 
	\begin{equation}\label{ffl}
	(-\Delta)^av = \mathscr{F}^{-1}\{|\xi|^{2a}\widehat{v}(\xi)\},
	\end{equation} where $\widehat{v}=\mathscr{F}v$ is the Fourier transform of $v$. It is a nonlocal operator since it does not preserve the support of $v$. 
	
	\noindent In \cite{GSU} it was shown, that one can uniquely recover the potential $q$ in $\Omega$ from the exterior measurements of the nolocal Cauchy data $\left(v|_{W}, (-\Delta)^av|_{\widetilde{W}}\right)$,  where $W,\widetilde{W}$ are non-empty open subsets of the exterior domain  $\Omega_e:=\mathbb{R}^n\setminus\overline{\Omega}$. 
	
	In this article, we are interested in the inverse problem of recovering the inhomogeneous nonlocal operator from the exterior Cauchy data. For instance, let us consider the  nonlocal equation  
	\[\left(-div(\gamma\nabla)\right)^au=0 \quad\mbox{in }\Omega,\quad 0<a<1.\] We address the question among several others of  whether the \textit{exterior Cauchy data} $\mathcal{C}^a_\gamma$ (cf. \eqref{eq:pc}) measured in the exterior domain $\Omega_e$, determines $\gamma$ in  $\Omega$.   
	
	
We formulate the problem for positive definite, selfadjoint second order elliptic operators.	 Let $\mathcal{L}$ be a second-order elliptic partial differential operator of the form
\begin{equation}\label{lopera}
\mathcal{L}:=-\sum_{jk=1}^n \frac{\partial}{\partial x_j}a_{jk}(x)\frac{\partial}{\partial x_k} - \mathrm{i}\sum_{j=1}^n\left(\frac{\partial}{\partial x_j}b_j(x) + b_j(x)\frac{\partial}{\partial x_j}\right) +c(x), \quad x\in\mathbb{R}^n
\end{equation}
where  $\mathrm{i} = (-1)^{1/2}$ and the coefficients $a_{jk}, b_j$ and $c$ are real-valued functions defined on $\mathbb{R}^n$. 
We assume that $\mathcal{L}$ is a self-adjoint, positive definite operator,
densely defined in $L^{2}(\mathbb{R}^{n})$. 
We consider the fractional
operator $\mathcal{L}^a$ in $\mathbb{R}^n$ with $a\in(0,1)$, and consider the nonlocal \textit{Calder\'{o}n}  problem associated to the fractional
operator $\mathcal{L}^{a}$ in a bounded domain $\Omega \subset\mathbb{R}^n$, We would like to determine the unknown coefficients $a_{jk}(x), b_j(x), c(x)$ from the exterior measurements of the \textit{Cauchy data} of the nonlocal elliptic equation $\mathcal{L}^{a}u=0$ in $\Omega$.  

  We briefly mention here that the fractional $\mathcal{L}^a$ can be defined using the  heat semi-group $\{e^{-t\mathcal{L}}\}_{t\geq 0}$ as 
  \begin{equation}\label{defLa}
  \mathcal{L}^a=\frac{1}{\Gamma(-a)}\int_{0}^{\infty}\frac{\left(e^{-t\mathcal{L}}-\mbox{Id}\right)}{t^{1+a}}\,dt
  \end{equation}
  where $\Gamma$ denotes the Gamma function.  It acts as a  bounded linear operator from  $H^a(\mathbb{R}^{n})$ to $H^{-a}(\mathbb{R}^{n})$ with its domain of definition $H^{2a}(\mathbb{R}^n)$, we refer Section \ref{sec2} for the details. For example, if $\mathcal{L}=(-\Delta)$, then the definition \eqref{defLa} coincides with the Fourier one \eqref{ffl} introduced above.  A similar definition was used in \cite{GLX} for the case that $\mathcal{L}$ has no first order terms.\\
\\
\paragraph{\bf Nonlocal equation}
 Let us consider
$u\in H^a(\mathbb{R}^{n})$ be a solution to the nonlocal Dirichlet problem
\begin{align}
\label{eq:nonlocalproblem}
\begin{split}
\mathcal{L}^au & = 0 \mbox{ in } \Omega,\\
supp\, u & \subseteq \overline{\Omega}\cup\overline{W}. 
\end{split}
\end{align}
\paragraph{\bf Nonlocal Cauchy data}
We define the nonlocal exterior \textit{partial Cauchy data} 
 $\mathcal{C}_{(W,\widetilde{W})}\subset H^a(W)\times H^{-a}(\widetilde{W})$ of the solution $u$ of \eqref{eq:nonlocalproblem} as 
	\begin{equation}\label{eq:pc} \mathcal{C}_{(W, \widetilde{W})} = \{ u|_{W},\, \mathcal{L}^au|_{\widetilde{W}}\}.
	\end{equation}
\\
At this point we introduce the problem we study here:   
\paragraph{\bf (A) Nonlocal inverse (exterior value) problem}
Does the exterior Cauchy data 
$\mathcal{C}_{(W, \widetilde{W})}$ determine 
 the coefficients $a_{jk},\, b_j,\, c$ uniquely in $\Omega$?\\
\\  
Let us recall the analogous local inverse problem here.
\paragraph{\bf (B) Local inverse (boundary value) problem} Is it possible to determine the coefficients $a_{jk},\, b_j,\, c$ of $\mathcal{L}$  in $\Omega$ from the associated boundary Cauchy data
\begin{equation}\label{eq:bc} \mathcal{C}_{\partial\Omega} = \{ v|_{\partial\Omega},\, \partial_\nu v|_{\partial\Omega}\}
\end{equation}
where $\nu= \sum_{j=1}^na_{jk}\nu_j$ is the usual co-normal vector on $\partial\Omega$, and
 $v\in H^1(\Omega)$ be the solution of the equation 
\begin{align}
\label{eq:localproblem}
\begin{split}
\mathcal{L}v & = 0 \mbox{ in } \Omega.
\end{split}
\end{align}
In this article, we prove that the above nonlocal inverse problem (A) can be reduced to the local inverse problem (B). Therefore, all the know results in local inverse problem (B) can be applied to solve the nonlocal inverse problem (A).

Let us now state precisely  our main result here. We assume that $\mathcal{L}$ in \eqref{lopera} for $n\geq2$, where $A(x)=(a_{jk}(x))$, $x\in\mathbb{R}^n$ is an $n\times n$ symmetric matrix satisfying the followings.
\begin{equation}\label{eq:ellipticity and symmetry condition}
\begin{cases}
a_{jk}=a_{kj}\mbox{ for all }1\leq j,k\leq n,\mbox{ and }\\
\Lambda^{-1}|\xi|^{2}\leq\sum_{j,k=1}^{n}a_{jk}(x)\xi_{i}\xi_{j}\leq\Lambda|\xi|^{2}\mbox{ for all }x,\xi\in\mathbb{R}^{n},\mbox{ for some }\Lambda>0
\end{cases}
\end{equation}
 and
\begin{equation}\label{eq:condi_bc} b_j\in W^{1,\infty}(\mathbb{R}^n)\cap\mathcal{E}^\prime(\overline{\Omega}) \mbox{ and }c\in L^\infty(\mathbb{R}^n)\cap\mathcal{E}^\prime(\overline{\Omega})
\end{equation}
such that $\mathcal{L}$ remains as self-adjoint positive definite operator in $L^2(\mathbb{R}^n)$ with its domain of definition $H^{2}(\mathbb{R}^n)$.  
Here is our main theorem. 
\begin{Theorem}\label{thm1}
Let $\mathcal{L}^{(1)}$, $\mathcal{L}^{(2)}$ be two  self-adjoint, positive definite, second-order elliptic differential operators as in \eqref{lopera} with the coefficients $a^{(l)}_{jk}$ satisfying \eqref{eq:ellipticity and symmetry condition}, and $b^{(l)}_j,\, c^{(l)}$ satisfying \eqref{eq:condi_bc} for $l=1,2$ respectively. 
	 Let $\Omega\subset \R^n,\, n\geq 2$ be some bounded non-empty open set. We assume that, $a^{(1)}_{jk}=a^{(2)}_{jk}=\delta_{jk}$ while restricted in $\Omega_e$, where $\delta_{jk}$ denotes the Kronecker delta. Let $W,\, \widetilde{W}\subset\Omega_e$ be two non-empty open subsets. Suppose that the exterior partial Cauchy data defined in \eqref{eq:pc} are same, i.e. $\mathcal{C}^{(1)}_{(W, \widetilde{W})}=\mathcal{C}^{(2)}_{(W, \widetilde{W})}$, for the two sets of solutions $\{u^{(1)}\in H^a(\R^n): (\mathcal{L}^{(1)})^a\,u^{(1)}=0 \mbox{ in } \Omega,\,supp\, u^{(1)}  \subseteq \overline{\Omega}\cup\overline{W},\,0<a<1\}$ and $\{u^{(2)}\in H^a(\R^n): (\mathcal{L}^{(2)})^a\,u^{(2)}=0 \mbox{ in } \Omega,\,supp\, u^{(2)}  \subseteq \overline{\Omega}\cup\overline{W}, \,0<a<1\}$.	
	Then the boundary Cauchy data defined in \eqref{eq:bc} are same, i.e. $\mathcal{C}^{(1)}_{\partial\Omega}=\mathcal{C}^{(2)}_{\partial\Omega}$, for the two sets of solutions $\{v^{(1)}\in H^1(\Omega): \mathcal{L}^{(1)}\,v^{(1)}=0 \mbox{ in } \Omega\}$ and $\{v^{(2)}\in H^1(\Omega): \mathcal{L}^{(2)}\,v^{(2)}=0 \mbox{ in } \Omega\}$.	
	\end{Theorem}
\subsection{Corollaries of Theorem \ref{thm1}}
We have reduced our nonlocal inverse problem into solving the local  inverse problem. The known results for the well studied local cases can be recalled here. We begin with the anisotropic case. Anisotropic conductivities depend on direction. Muscle tissue in the human body is
an important example of an anisotropic conductor. For instance cardiac muscle has a
conductivity of 2.3 mho in the transverse direction and 6.3 in the longitudinal direction.
\subsection*{Anisotropic case}
Let us consider $\mathcal{L}$ in \eqref{lopera} as
\begin{equation}\label{gm12}
\mathcal{L}_0:=-\sum_{jk=1}^n \frac{\partial}{\partial x_j}a_{jk}(x)\frac{\partial}{\partial x_k}, \qquad x\in\mathbb{R}^n
\end{equation} 
where $A(x)=(a_{jk}(x))$, $x\in\mathbb{R}^n$ is an $n\times n$ symmetric matrix satisfying the ellipticity and boundedness  criteria given in \eqref{eq:ellipticity and symmetry condition}.
Let $\Omega$ as in Theorem \ref{thm1}.  We assume that, $a_{jk}$ are smooth functions on $\Omega$ and $a_{jk}=\delta_{jk}$  in $\Omega_e$. 
 By Theorem \ref{thm1}, it reduces into the following inverse problem that whether one can determine $A(x)$ in $\Omega$ by knowing the boundary Cauchy data 
 \[\mathcal{C}^A_{\partial\Omega} = \Big\{ v|_{\partial\Omega},\, \sum_{j=1}^n\nu_ja_{jk}\frac{\partial v}{\partial x_k}\big|_{\partial\Omega}\Big\}
 \]
 where $v\in H^1(\Omega)$ be the solution of the anisotropic conductivity equation 
 \[
 \sum_{jk=1}^n \frac{\partial}{\partial x_j}\left(a_{jk}(x)\frac{\partial}{\partial x_k}v\right)=0 \quad\mbox{in }\Omega.
 \]
Unfortunately, $\mathcal{C}^A_{\partial\Omega}$ does not determine $A$ in $\Omega$ uniquely. This observation is due to L. Tartar (see \cite{KVC} for instance). Let 
\[\mathbb{F}: \overline{\Omega} \mapsto \overline{\Omega}\] be a $C^\infty$ diffeomorphism with $\mathbb{F}|_{\partial\Omega} =Id$
where Id denotes the identity map. We define the push forward of $A$ as
\begin{equation}\label{Atilda} \mathbb{F}_{*}A= \Big(\frac{(D\mathbb{F})^\top\circ A \circ (D\mathbb{F})}{|{\rm det }(D\mathbb{F})|}\Big)\circ\mathbb{F}^{-1},  \end{equation}
where $D\mathbb{F}$ denotes the (matrix) differential of $\mathbb{F}$, $(D\mathbb{F})^\top$ its transpose and the composition in \eqref{Atilda} is to be interpreted as multiplication of matrices. Then we have
\[ \mathcal{C}^A_{\partial\Omega}= \mathcal{C}^{\mathbb{F}_{*}A}_{\partial\Omega}.\]
This shows we have then a large
number of conductivities with the same Cauchy data any change of variables of $\Omega$ that
leaves the boundary fixed gives rise to a new conductivity with the same electrostatic boundary measurements. The question is then whether this is the only obstruction to unique identifiability of the conductivity. 

In two dimensions this has been shown
for $L^\infty(\Omega)$ conductivities in \cite{APL}. This is done by reducing the anisotropic problem
to the isotropic one by using isothermal coordinates \cite{SYL90} and using the Astala and
P\"{a}iv\"{a}rinta's result in the isotropic case \cite{AP}. 
Here is our result in two dimension.
\begin{Theorem}\label{thm20}
	Let $n=2$.	Let $\mathcal{L}^{(1)}_0$, $\mathcal{L}^{(2)}_0$ are as in \eqref{gm12} with $A^{(l)}(x)=(a^{(l)}_{jk}(x))$ satisfying \eqref{eq:ellipticity and symmetry condition} for $l=1,2$ respectively. Let $\Omega, W,\, \widetilde{W}$ are as in Theorem \ref{thm1}.  We assume that, $a_{jk}^{(1)}=a_{jk}^{(2)}=\delta_{jk}$  while restricted in $\Omega_e$. Suppose  $\mathcal{C}^{(1)}_{(W, \widetilde{W})}=\mathcal{C}^{(2)}_{(W, \widetilde{W})}$ for the two sets of solutions $\{u^{(1)}\in H^a(\R^n): (\mathcal{L}_0^{(1)})^a\,u^{(1)}=0 \mbox{ in } \Omega,\, supp\, u^{(1)}  \subseteq \overline{\Omega}\cup\overline{W},\,0<a<1\}$ and $\{u^{(2)}\in H^a(\R^n): (\mathcal{L}_0^{(2)})^a\,u^{(2)}=0 \mbox{ in } \Omega,\,supp\, u^{(2)}  \subseteq \overline{\Omega}\cup\overline{W}, \,0<a<1\}$. 	
	Then there exists a smooth, invertible map $\mathbb{F}:\overline{\Omega}\mapsto\overline{\Omega}$, with ${\rm det}(D\mathbb{F})(x)$, ${\rm det}(D\mathbb{F}^{-1})(x) \ge C >0 $ in $\Omega$, and $\mathbb{F} = I$ on $\partial\Omega$, such that 
	\[A^{(2)}=\mathbb{F}_{*}A^{(1)} \mbox{ in } \Omega.\]
\end{Theorem}

\subsection*{Recovering a Riemannian metric}
In three dimensions, as was pointed out in \cite{LU89}, this is a problem of geometrical nature and makes sense for general compact Riemannian manifolds with boundary. Let $(M, g)$ be a compact Riemannian manifold
with boundary. The Laplace-Beltrami operator associated to the metric $g$ is given in
local coordinates by
\begin{equation}\label{Dg} (-\Delta_g) =- \frac{1}{\sqrt{g}}\sum_{j,k=1}^n\frac{\partial}{\partial x_j}\left(\sqrt{g}g^{jk}\frac{\partial}{\partial x_k}\right) \end{equation}
where $(g^{jk})$ is the matrix inverse of the matrix $(g_{jk})$.

The inverse problem is whether one can recover $g$ from the boundary Cauchy data 
\[
\mathcal{C}^g_{\partial M} = \Big\{ v|_{\partial M},\, \sum_{j=1}^n\nu_j\,g^{jk}\frac{\partial v}{\partial x_k}\,\sqrt{det\,g}\,\big|_{\partial M}\Big\}
\]
where $v\in H^1(M)$ be the solution of 
\[
(-\Delta_g)v=0 \quad\mbox{in }M.
\]
We have
\begin{equation}\label{gM}
 \mathcal{C}^g_{\partial M} =\mathcal{C}^{\mathbb{F}^{*}g}_{\partial M}
\end{equation}
where $\mathbb{F}$ is any $C^\infty$ diffeomorphism of $M$ which is the identity on the boundary. As
usual $\mathbb{F}^{\ast}g$ denotes the pull back of the metric $g$ by the diffeomorphism $\mathbb{F}$. 

In the case that $M$ (say $M=\Omega$) is an open, bounded subset of $\mathbb{R}^n$  with smooth boundary, then it is easy to see (\cite{LU89}) that for $n \geq 3$ the anisotropic Calder\'{o}n problem and the problem of recovering the metric from the Cauchy data are equivalent. Namely, we have
\[
\mathcal{C}^A_{\partial M} = \mathcal{C}^g_{\partial M}
\]
where
\[
g_{jk} = (det\, A)^{1/(n-2)}\,a_{jk}^{-1},\quad a_{jk} = (det\, g)^{1/2}\,g_{jk}^{-1}.
\]
 Lassas and Uhlmann (\cite{LUrm}, see also \cite{LTU}) proved
that \eqref{gM} is the only obstruction to unique identifiability of the conductivity for real-analytic manifolds in dimension $n \geq 3$. 
\subsection{$a$-harmonic functions}
Let us consider the Riemannian manifold $(\mathbb{R}^n,g)$ with the metric $g$ as
\begin{equation}\label{rm}g_{jk} = (det\, A)^{1/(n-2)}a_{jk}^{-1}\quad\mbox{ in }\mathbb{R}^n.\end{equation}
Then the operator $\mathcal{L}_0$ in \eqref{gm12} is the Laplace-Beltrami operator $(-\Delta_g)$ (cf. \eqref{Dg}) in $(\mathbb{R}^n,g)$. Let $\Omega, W, \, \widetilde{W}$ are as in Theorem \ref{thm1}.  Let us consider the $a$-harmonic functions 
$u\in H^a(\mathbb{R}^{n})$ as
\begin{align}
\label{eq:nonlocalproblem_g}
\begin{split}
(-\Delta_g)^au & = 0 \mbox{ in } \Omega,\\
supp\, u & \subseteq \overline{\Omega}\cup\overline{W}. 
\end{split}
\end{align}
We define the exterior (partial) \textit{Cauchy data} 
$\mathcal{C}_{(W, \widetilde{W})}\subset H^a(W)\times H^{-a}(\widetilde{W})$ as 
\begin{equation}\label{eq:pc_g} \mathcal{C}_{(W, \widetilde{W})} = \{ u|_{W},\, (-\Delta_g)^au|_{\widetilde{W}}\}.
\end{equation}
Here we state our result. 
\begin{Theorem}\label{thm2}
		Let $n\geq 3$. Let $(\R^n, g^{(l)})$, $l=1,2$ be two Riemannian manifolds defined as in \eqref{rm}. Let $\Omega, W, \widetilde{W}$ are as in Theorem \ref{thm1}. We assume that, $g^{(1)}=g^{(2)}=g_0$ while restricted in $\Omega_e$, where $(\R^n,g_0)$ stands for the standard Euclidean space. We further assume $(\Omega,g^{(l)})$, $l=1,2$ are real analytic, and connected. Suppose that the exterior partial Cauchy data defined in \eqref{eq:pc_g} are the same i.e. $\mathcal{C}^{(1)}_{(W, \widetilde{W})}=\mathcal{C}^{(2)}_{(W, \widetilde{W})}$, for the two sets of $a$-harmonic functions $\{u\in H^a(\R^n): (-\Delta_{g^{(1)}})^au=0 \mbox{ in } \Omega\}$ and $\{v\in H^a(\R^n): (-\Delta_{g^{(2)}})^av=0 \mbox{ in } \Omega\}$, $0<a<1$, corresponding to the two different metrics $g^{(1)},g^{(2)}$ respectively.
	Then there exists a real analytic, invertible map $\mathbb{F}:\overline{\Omega}\mapsto\overline{\Omega}$, with ${\rm det}(D\mathbb{F})(x)$, ${\rm det}(D\mathbb{F}^{-1})(x) \ge C >0 $ in $\Omega$, and $\mathbb{F} = I$ on  $\partial\Omega$, such that $g^{(2)}=\mathbb{F}^{*}g^{(1)}$ in $\Omega$.
\end{Theorem}
At this point we ask, similar to the local case, do we always have the invariance through change of variable in nonlocal inverse problem? The answer turns out to be yes.  In nonlocal case  also we can not except the full recovery of the metric, and it is subjected to the invariance under the change of variables by globally ($\mathbb{R}^n$) defined diffeomorphisms which are the identity in the exterior. This is known as transformation optics. We postpone the discussion on this topic to Subsection \ref{top} and refer Theorem \ref{thm6} for the precise statement.      
\subsection*{Isotropic case}
Let us turn into more specific case by considering $A^{(l)}$, $l=1,2$ are isotropic matrices, i.e. $(A^{(l)})_{jk}=a^{(l)}\delta_{jk}$. Further with the right regularity assumptions on $a^{(l)}\big|_{\Omega}$ while $a^{(l)}\big|_{\Omega_e}=1$, we have the desired uniqueness of $a^{(1)}=a^{(2)}$. In two dimension, as we have mentioned above the work of \cite{AP} gives the result  
for bounded conductivities, so in particular \eqref{eq:ellipticity and symmetry condition} is enough to have the following uniqueness result. 
\begin{Theorem}\label{thm30}
	Let $n=2$. Let $A^{(1)}, A^{(2)}$ are isotropic matrices in Theorem \ref{thm20}. Then $\mathbb{F}$ becomes the identity in $\Omega$, that $A^{(1)}=A^{(2)}$ in $\Omega$. 
\end{Theorem}
In three and higher dimensions, we present our result with $g^{(l)}\big|_{\Omega}\in C^2(\overline{\Omega})$ as first resolved by Sylvester and Uhlmann in \cite{SU}. With several intermediate improvements (see \cite{BT03,GLU2003,PPU03}) few years ago this has been extended to $C^1$ regularity by Haberman and Tataru in \cite{HT13} and later into critical (with respect to unique continuation) $W^{1,n}$ space by Haberman in \cite{HB15} for $n=3,4$. 
\begin{Theorem}\label{thm3}
	Let $n\geq 3$. Let $g^{(1)},g^{(2)}$ are isotropic metrics in Theorem \ref{thm2}. We further assume $g^{(l)}\big|_{\overline{\Omega}}\in C^2(\overline{\Omega})$. Then $\mathbb{F}$ becomes the identity in $\Omega$, that $g^{(1)}=g^{(2)}$ in $\overline{\Omega}$. 
\end{Theorem}
 \subsection*{Recovering the lower order terms}
Let us consider $\mathcal{L}$ in \eqref{lopera} as
\begin{equation}\label{msch}
\mathcal{L}_1:=-\Delta - \mathrm{i}\sum_{j=1}^n\left(\frac{\partial}{\partial x_j}b_j(x) + b_j(x)\frac{\partial}{\partial x_j}\right) +c(x), \quad x\in\mathbb{R}^n
\end{equation}
where the coefficients $b_j, c$ satisfy \eqref{eq:condi_bc}. 
Since \cite{SUN1, SUN2}, inverse boundary value problems for
first order perturbations of the Laplacian have been extensively studied, usually
in the context of magnetic Schr\"{o}dinger operators, see \cite{NSU, CNS, KU14}.
\begin{Theorem}\label{thm4}
		Let $\mathcal{L}^{(1)}_1$, $\mathcal{L}^{(2)}_1$ are as in \eqref{msch} with $b^{(l)}_j,\, c^{(l)}$ satisfying \eqref{eq:condi_bc} for $l=1,2$ respectively. Let $\Omega, W,\, \widetilde{W}$ are as in Theorem \ref{thm1}. Suppose  $\mathcal{C}^{(1)}_{(W, \widetilde{W})}=\mathcal{C}^{(2)}_{(W, \widetilde{W})}$ for the two sets of solutions $\{u^{(1)}\in H^a(\R^n): (\mathcal{L}_1^{(1)})^a\,u^{(1)}=0 \mbox{ in } \Omega,\, supp\, u^{(1)}  \subseteq \overline{\Omega}\cup\overline{W},\,0<a<1\}$ and $\{u^{(2)}\in H^a(\R^n): (\mathcal{L}_1^{(2)})^a\,u^{(2)}=0 \mbox{ in } \Omega,\,supp\, u^{(2)}  \subseteq \overline{\Omega}\cup\overline{W}, \,0<a<1\}$.	
		Then $Curl\, b^{(1)}= Curl\, b^{(2)}$  and $c^{(1)}-(b^{(1)})^2 = c^{(2)}-(b^{(2)})^2$ in $\Omega$. 
	\end{Theorem}
	As a  particular case that if $\mathcal{L}$ in \eqref{lopera} has only the zeroth order term, say \begin{equation}\label{poten}
	\mathcal{L}_2:=-\Delta +c(x), \quad x\in\mathbb{R}^n
	\end{equation}
	where the coefficients $c$ satisfy \eqref{eq:condi_bc}. Then we have the following result. There are numerous work has been done centered around the Schr\"{o}dinger equation, see \cite{UG} for a survey.
	\begin{Theorem}\label{thm5}
		Let $\mathcal{L}^{(1)}_2$, $\mathcal{L}^{(2)}_2$ are as in \eqref{poten} with $c^{(l)}$ satisfying \eqref{eq:condi_bc} for $l=1,2$ respectively. Let $\Omega, W,\, \widetilde{W}$ are as in Theorem \ref{thm1}, and suppose  $\mathcal{C}^{(1)}_{(W, \widetilde{W})}=\mathcal{C}^{(2)}_{(W, \widetilde{W})}$ for the two sets of solutions $\{u^{(1)}\in H^a(\R^n): (\mathcal{L}_2^{(1)})^a\,u^{(1)}=0 \mbox{ in } \Omega,\, supp\, u^{(1)}  \subseteq \overline{\Omega}\cup\overline{W},\,0<a<1\}$ and $\{u^{(2)}\in H^a(\R^n): (\mathcal{L}_2^{(2)})^a\,u^{(2)}=0 \mbox{ in } \Omega,\,supp\, u^{(2)}  \subseteq \overline{\Omega}\cup\overline{W}, \,0<a<1\}$.	
		Then $c^{(1)}= c^{(2)}$ in $\Omega$. 
	\end{Theorem}
\noindent
This completes the discussion of the corollaries of Theorem \ref{thm1} and their proofs.

	Finally, we would like to point out that the results obtained in this paper can
	be viewed as generalizations of the  fractional Calder\'{o}n problem which begins with the article \cite{GSU}, subsequent development includes results like low regularity and stability \cite{RS1,RS4}, matrix coefficients \cite{GLX}, variable coefficient \cite{COVI1}, reconstruction from single measurement \cite{GRSU}, shape detection \cite{HL2}, local and nonlocal lower order perturbation \cite{CLR, LiLi1, BGU}, recovery from the boundary response \cite{TG} etc.  See also the survey article \cite{SALS} and the references therein. These nonlocal problems are motivated by the various modeling ranging from diffusion process \cite{AFMR} to finance \cite{SCW}, image processing \cite{GGOS}, biology \cite{MAVE} etc. See \cite{BCVE, ROSS, JBRW, CGRA} for further references.

	This paper is organized as follows. In
	Section \ref{sec2} we review the functional framework namely the spectral theory and heat semigroup approach to define the fractional operators. Following that, we discuss the direct problem for the fractional equation and introduce the exterior Cauchy data.  In Section \ref{sec3}, we discuss the inverse problem and complete the proof of the Theorem \ref{thm1}. The proof is divided into several propositions and lemmas directed into establishing the claim. In the final Section \ref{sec4} we discuss the gauge invariance of the nonlocal inverse problem (cf. (A)) and establish similar to the local case in the nonlocal case also the unique recovery of the anisotropic matrices is subjected to the global change of variables or the transformation optics phenomena. Precise statement can be found in Theorem \ref{thm6}.

	We became aware of the preprint \cite{Ali}
	when the writing of this paper was being finished. This article
	considers the source to solution map associated with the fractional
	Laplace-Beltrami operator on a closed Riemannian manifold under some additional conditions.

	\subsection*{Acknowledgement}
	The research of T.G. is supported by the Collaborative Research Center, membership no. 1283, Universit\"{a}t Bielefeld. G.U. was partly supported by NSF a Walker Family Endowed Professorship at UW and a Si Yuan Professorship at IAS, HKUST. G.U. would like to thank Matti Lassas and Mikko Salo for the opportunity to present the results of this article in the conference  ``Inverse problems and nonlinearity'' in August, 2021. 

\section{Functional framework}\label{sec2}
\subsection{Spectral Theory $\&$ Heat-Semigroup \label{sec:SpectralTheory} }
Let us briefly present the spectral theory approach (\cite{Riesz,Rudin2}) to define the operator $\mathcal{L}^a$. 
Let $\mathcal{L}$ be any non-negative definite and self-adjoint operator
densely defined in $L^{2}(\mathbb{R}^{n})$. $L^2(\mathbb{R}^n)$ consists of square integrable complex-valued functions on $\mathbb{R}^n$ equipped with the inner product $\langle f, g\rangle = \int_{\mathbb{R}^n}f(x)\overline{g(x)}\,dx$. 
Let $\phi$ be a real-valued measurable function defined on the spectrum
of $\mathcal{L}$. Then one defines $\phi(\mathcal{L})$
be a self-adjoint operator in $L^{2}(\mathbb{R}^{n})$ as
\[
\phi(\mathcal{L}):=\int_{0}^{\infty}\phi(\lambda)\,dE_{\lambda},
\]
where $\{E_{\lambda}\}$ is the spectral resolution of $\mathcal{L}$
and each $E_{\lambda}$ is a projection in $L^{2}(\mathbb{R}^{n})$
(see for instance, \cite{grigoryan2009heat}). The domain of $\phi(\mathcal{L})$
is given by 
\begin{equation}\label{dfdm}
\mathrm{Dom}(\phi(\mathcal{L}))=\left\{ f\in L^{2}(\mathbb{R}^{n});\int_{0}^{\infty}|\phi(\lambda)|^{2}\,d\|E_{\lambda}f\|^{2}<\infty\right\} .
\end{equation}
The linear operator $\phi(\mathcal{L}):\mathrm{Dom}(\phi(\mathcal{L}))\rightarrow L^{2}(\mathbb{R}^{n})$
is understood, via the Riesz representation theorem, in the following
sense, 
\[
\left\langle \phi(\mathcal{L})f,g\right\rangle :=\int_{0}^{\infty}\phi(\lambda)\,d\langle E_{\lambda}f,g\rangle,\quad f\in\mathrm{Dom}(\phi(\mathcal{L})),\ g\in L^{2}(\mathbb{R}^{n}).
\]
Now we consider the particular case $\phi(\lambda)=\lambda^a$ in order to define the fractional operator $\mathcal{L}^a$, $a\in (0,1)$.
We also note that  \[\lambda^a=\dfrac{1}{\Gamma(-a)}\int_{0}^{\infty}(e^{-t\lambda}-1)t^{-1-a}\,dt, \quad a\in(0,1),\]
where $\Gamma(-a):=-\Gamma(1-a)/a$, and $\Gamma$
is the Gamma function. We have 
\begin{equation}\label{spec_heat}
\mathcal{L}^{a}:=\int_{0}^{\infty}\lambda^{a}\,dE_{\lambda}=\frac{1}{\Gamma(-a)}\int_{0}^{\infty}\left(e^{-t\mathcal{L}}-\mbox{Id}\right)\,\frac{dt}{t^{1+a}},\quad a\in (0,1)
\end{equation}
where $e^{-t\mathcal{L}}$ given by 
\begin{equation}\label{eq:heat-semigroup}
e^{-t\mathcal{L}}:=\int_{0}^{\infty}e^{-t\lambda}\,dE_{\lambda}
\end{equation}
is a bounded self-adjoint operator in $L^{2}(\mathbb{R}^{n})$ for
each $t\ge0$. The operator family $\{e^{-t\mathcal{L}}\}_{t\ge0}$
is called the heat semigroup associated with $\mathcal{L}$ (cf. \cite{Pazy}).
\subsection{Sobolev Spaces}
Let us introduce few spaces to work with. We follow the notations of \cite{McLean}. Let
$H^s(\mathbb{R}^n)$ denotes the fractional Sobolev space over $\mathbb{R}^n$ of order $s\in\mathbb{R}$:
\begin{equation}\label{hs}
 H^s(\R^n) := \{ u\in \mathcal{S}^\prime(\mathbb{R}^n)\,\, |\,\, (1+|\xi|^2)^{s/2}\widehat{u}(\xi)\in L^2(\mathbb{R}^n). \} \end{equation}
where 
$\mathcal{S}^\prime(\mathbb{R}^n)$ denotes the space of tempered distributions in $\mathbb{R}^n$, and $\widehat{\cdot}$ denotes the Fourier transform.  

\noindent
Let  $U\subset \R^n$ be an open set. We define
\begin{align*}
H^{s}(U)&:= \{u|_{U}: \, u \in H^{s}(\R^n)\},\\
\widetilde{H}^s(U)&:=\mbox{closure of $C_c^{\infty}(U)$ in $H^{s}(\R^n)$}.
\end{align*}
Let $s\in(0,1)$, we define
\begin{equation}\label{eq:NormHs}
\|u\|^2_{H^{s}(\mathbb{R}^{n})}:=\|u\|^2_{L^{2}(\mathbb{R}^{n})}+\int_{\R^n}\int_{\R^n}\frac{\left|u(x)-u(z)\right|^{2}}{|x-z|^{n+2s}}\,dx\,dz
\end{equation}
and
\[
\|u\|_{H^{s}(U)}:=\inf\left\{ \|w\|_{H^{s}(\mathbb{R}^{n})};\,w\in H^{s}(\mathbb{R}^{n})\mbox{ and }w|_{U}=u\right\} .
\]
\subsection{The fractional operator $\mathcal{L}^{a}$, $0<a<1$}
It is known that the 
 operator $\mathcal{L}$ introduced in  \eqref{lopera}-\eqref{eq:ellipticity and symmetry condition}-\eqref{eq:condi_bc} with the domain  
$\mathrm{Dom}(\mathcal{L})=H^{2}(\mathbb{R}^{n})$
is the maximal extension such that $\mathcal{L}$ is self-adjoint positive definite operator densely defined in $L^{2}(\mathbb{R}^{n})$. Moreover, by the definition in \eqref{dfdm}, it follows that $\mathrm{Dom}(\mathcal{L}^{a})=H^{2a}(\mathbb{R}^{n})$. Next, we would like to
extend the domain of definition of $\mathcal{L}^{a}$  to $H^{a}(\mathbb{R}^n)$, using heat
kernels and their estimates, in order to solve the direct problem \eqref{eq:nonlocalproblem}.

It is also known that for $\mathcal{L}$ satisfying \eqref{lopera}-\eqref{eq:ellipticity and symmetry condition}-\eqref{eq:condi_bc},
the bounded operator $e^{-t\mathcal{L}}$ given in \eqref{eq:heat-semigroup}
admits a symmetric (heat) kernel $p_{t}(x,z)$ (cf. \cite{GRAL}).
In other words, one has for any $t\in\mathbb{R}_{+}:=(0,\infty)$
and any $f\in L^{2}(\mathbb{R}^{n})$ that 
\begin{equation}\label{eq:hk}
\left(e^{-t\mathcal{L}}f\right)(x)=\int_{\mathbb{R}^{n}}p_{t}(x,z)f(z)\,dz,\quad x\in\mathbb{R}^{n}.
\end{equation}
Since we have assumed $\mathcal{L}$ (cf. \eqref{lopera}) is a positive definite operator i.e. Spec\, $\mathcal{L}\subset (0,\infty)$, so $\mathcal{L}$ can be compared with  $\widetilde{\mathcal{L}}=-\sum_{jk=1}^n \frac{\partial}{\partial x_j}\widetilde{a}_{jk}(x)\frac{\partial}{\partial x_k}$ in $\mathbb{R}^n$, where 
 $\widetilde{A}(x)=(\widetilde{a}_{jk}(x))$, $x\in\mathbb{R}^{n}$ is an $n\times n$
symmetric matrix satisfying the ellipticity condition \eqref{eq:ellipticity and symmetry condition}.
Then the (heat) kernel $p_{t}(\cdot,\cdot)$ (cf. \eqref{eq:hk}) for $\mathcal{L}$ admits the following estimates (see \cite{CoTh}) 
\begin{equation}\label{eq:pointwise estimates for p_t}
c_1\, \Big(\frac{1}{4\pi\,t}\Big)^{\frac{n}{2}} e^{-\frac{d_1\,|x-z|^2}{4t}}\leq p_t(x,z) \leq c_2\,\Big(\frac{1}{4\pi\,t}\Big)^{\frac{n}{2}} e^{-\frac{d_2\,|x-z|^2}{4t}}, \quad x, z\in\mathbb{R}^n
\end{equation}
for some $c_1,c_2, d_1,d_2>0$.

Then from \eqref{spec_heat} (see also \cite{CSt}) we write 
 for $f,g\in\mathrm{Dom}(\mathcal{L}^{a})$:  
\begin{equation}\label{eq:LsIntegral0}
\langle\mathcal{L}^{a}f,g\rangle=\frac{1}{2\Gamma(-a)}\int_{0}^{\infty}\int_{\mathbb{R}^{n}}\int_{\mathbb{R}^{n}}(f(x)-f(z))(g(x)-g(z))p_{t}(x,z)\,dx\,dz\,\frac{dt}{t^{1+a}}.
\end{equation}
Let us define
\begin{equation}\label{eq:kernel}
\mathcal{K}_{a}(x,z):=\frac{1}{\Gamma(-a)}\int_{0}^{\infty}p_{t}(x,z)\,\frac{dt}{t^{1+a}}.
\end{equation}
Thanks to \eqref{eq:pointwise estimates for p_t} it
enjoys the following pointwise estimate 
\begin{equation}\label{eq:pointwise estimate for kernel K}
\frac{C_{1}}{|x-z|^{n+2s}}\leq\mathcal{K}_{a}(x,z)=\mathcal{K}_{a}(z,x)\leq\dfrac{C_{2}}{|x-z|^{n+2s}},\quad x,z\in\mathbb{R}^{n},
\end{equation}
for $C_{1}, C_{2}>0$. Hence it is seen
by recalling the norm \eqref{eq:NormHs} of $H^{a}(\mathbb{R}^{n})$
that for any $f,g\in H^{s}(\mathbb{R}^{n})$, the right hand side
(RHS) of \eqref{eq:LsIntegral0} extends the definition of $\mathcal{L}^{a}$
from $\mathrm{Dom}(\mathcal{L}^{a})$ to $H^{a}(\mathbb{R}^{n})$
in the following distributional sense 
\begin{equation}\label{eq:integral represent for nonlocal}
\langle\mathcal{L}^{a}f,g\rangle:=\frac{1}{2}\int_{\mathbb{R}^{n}}\int_{\mathbb{R}^{n}}(f(x)-f(z))(g(x)-g(z))\mathcal{K}_{a}(x,z)\,dx\,dz
\end{equation}
with satisfying
\begin{equation}\label{eq:LsBoundedHs}
\left|\langle\mathcal{L}^{a}f,g\rangle\right|\le C\|f\|_{H^{a}(\mathbb{R}^{n})}\|g\|_{H^{a}(\mathbb{R}^{n})},\quad f,g\in H^{a}(\mathbb{R}^{n}).
\end{equation}
Thus, the definition \eqref{eq:integral represent for nonlocal} gives
a bounded linear operator 
\[
\mathcal{L}^{a}:H^{a}(\mathbb{R}^{n})\longrightarrow H^{-a}(\mathbb{R}^{n}).
\]
In particular, by simply using the symmetry $\mathcal{K}_{a}(x,z)=\mathcal{K}_{a}(z,x)$, one 
concludes $\mathcal{L}^{a}$ is self-adjoint, it is positive definite and given by
\begin{equation}\label{eq:PV for L^s}
\begin{split}\left(\mathcal{L}^{a}f\right)(x)= & \lim_{\epsilon\to0^{+}}\int_{|x-z|>\epsilon}(f(x)-f(z))\mathcal{K}_{a}(x,z)dz,\quad f\in H^{a}(\mathbb{R}^{n}).\end{split}
\end{equation}

\subsection{Dirichlet problems for $\mathcal{L}^{a}$}
Here we discuss the solvability of the direct problem \eqref{eq:nonlocalproblem}.
\subsection*{Well-Posedness}
Let $\Omega\subset \R^n$ be a bounded open set, and we denote $\Omega_e:=\R^n\setminus\overline{\Omega}$. 
Let $f\in H^a(\R^n)$ and  $u\in H^a(\R^n)$ be the solution of 
\begin{align}
\label{gm31}
\begin{split}
\mathcal{L}^au & = 0 \mbox{ in } \Omega,\\
u & = f \mbox{ in }\Omega_e.
\end{split}
\end{align}
The well-posedness of the above problem can be seen as follows: we define the bi-linear form $B: H^a(\R^n)\times H^a(\R^n)\mapsto \mathbb{C}$:
\[\begin{aligned}
B(u,w)&:=\langle \mathcal{L}^{a}u, w\rangle_{L^2(\R^n)} , \quad u,w \in H^{a}(\R^n).\\
&=\int_{\mathbb{R}^{n}}\int_{\mathbb{R}^{n}}(u(x)-u(z))(w(x)-w(z))\mathcal{K}_{a}(x,z)\,dx\,dz
\end{aligned}
\]
Then, for any $f\in H^a(\R^n)$ the problem \eqref{gm31} is well-posed in the sense that there exists a unique solution $u \in H^{a}(\R^n)$ with satisfying
\[
B(u,w) = 0 \mbox{ for all } w \in \widetilde{H}^a(\Omega),
\]
and $u-f \in \widetilde{H}^a(\Omega)$.
Moreover, there exists a constant $C>0$ independent of $u$ and $f$, such that
\begin{align*}
\|u\|_{H^{a}(\R^n)} \leq C \|f\|_{H^{a}(\R^n)}.
\end{align*}
The proof follows as in the works \cite{GLX,GSU}. 

Next we define the \textit{Cauchy data} $\mathcal{C}_{\Omega_e}$ in the exterior domain $\Omega_e$ as
\begin{equation}\label{cd}
\mathcal{C}_{\Omega_e} = \{ u|_{\Omega_e}, \mathcal{L}^au|_{\Omega_e}\}
\subset H^a(\Omega_e)\times H^{-a}(\Omega_e),\end{equation}
where $u$ solves \eqref{gm31}.

We also introduce the \textit{partial Cauchy data}.  Let  $ W,\, \widetilde{W} \subset \Omega_e$  be two non-empty open sets. Let us choose $f\in \widetilde{H}^a(W)$ in \eqref{gm31}, i.e. $supp\, u \subseteq \overline{\Omega}\cup \overline{W}$.  
We define the partial Cauchy data 
$\mathcal{C}_{(W, \widetilde{W})}$ as 
\begin{equation}\label{pc} \mathcal{C}_{(W, \widetilde{W})} = \{ u|_{W}, \mathcal{L}^au|_{\widetilde{W}}\} \subset H^a(W)\times H^{-a}(\widetilde{W}).
\end{equation}
\section{Inverse problems}\label{sec3}
 Let $\mathcal{L}^{(1)}$, $\mathcal{L}^{(2)}$ be two  self-adjoint, positive definite, second-order elliptic differential operators as in \eqref{lopera} with the coefficients $a^{(l)}_{jk}$ satisfying \eqref{eq:ellipticity and symmetry condition}, and $b^{(l)}_j,\, c^{(l)}$ satisfying \eqref{eq:condi_bc} for $l=1,2$ respectively. 
 Let $\Omega\subset \R^n,\, n\geq 2$ be some bounded non-empty open set. We assume that, $a^{(1)}_{jk}=a^{(2)}_{jk}=\delta_{jk}$ while restricted in $\Omega_e$. So what follows  $\mathcal{L}^{(1)}\Big|_{\Omega_e}=\mathcal{L}^{(2)}\Big|_{\Omega_e}=(-\Delta)$.  
   \subsection{Non-local inverse problem}\label{nlip}
   Let $f\in H^a(\R^n)$ and $u^{(l)}_f \in H^a(\R^n)$ be the unique solution of 
   \begin{equation}\label{fgk}\begin{cases} (\mathcal{L}^{(l)})^a\, u^{(l)}_f = 0 \mbox{ in }\Omega\\
   u^{(l)}_f=f \mbox{ in }\Omega_e.
   \end{cases}\qquad l=1,2, \quad 0<a<1
   \end{equation}
   Along our hypothesis in Theorem \ref{thm1}, let us assume
   \begin{equation}\label{L12}
   (\mathcal{L}^{(1)})^a\, u^{(1)}_f\Big|_{\widetilde{W}}=(\mathcal{L}^{(2)})^a\, u^{(2)}_f\Big|_{\widetilde{W}}
   \end{equation}
    for some non-empty open set $\widetilde{W}\subset\Omega_e$.

   Let us recall the heat kernel (cf. \eqref{eq:hk}) as $p^{(l)}_t(x,y)=H(t)p^{(l)}(x, y, t)$, where $H(\cdot)$ denotes the Heaviside function in $\R$, and $p(\cdot, \cdot, \cdot)$ solves
   \begin{equation}\label{hkernel2} \partial_tp^{(l)}(x,y,t) -\mathcal{L}^{(l)}p^{(l)}(x,y,t)=\delta(x-y,t) \quad\mbox{in }\R\times\R^n,\quad l=1,2 \end{equation}
   with satisfying (see \cite{Davies})
   \begin{equation}\begin{cases}\label{properties2} 
   p^{(l)}(\cdot,\cdot,\cdot) \in C^\infty (\R^n\times\R^n\times (0,\infty )),\,\, p^{(l)}(x,y,\cdot)=p^{(l)}(y,x,\cdot)\\[2pt]
   \int_{\R^n}p^{(l)}(x,y,t)\, dx =1, \quad t>0\\[2pt]
   p^{(l)}(x,y,t) \rightharpoonup \delta_y(x) \mbox{ as $t\to 0+$ in }\mathcal{D}^\prime(\R^n).
   \end{cases} \qquad l=1,2
   \end{equation}
   Let $u^{(l)}_f\in H^a(\mathbb{R}^n)$
   and 
   \[U^{(l)}= U^{(l)}(x,t)= \int_{\R^n}p^{(l)}_t(x,y)u^{(l)}_f(y)\,dy\,\,\, \in C\left([0, \infty); H^a(\R^n)\right)\] be the unique solution of
   \begin{equation}\label{kh}
   \begin{cases}
   \partial_t U^{(l)} = \mathcal{L}^{(l)}\, U^{(l)} \quad\mbox{in }\mathbb{R}^n\times (0,\infty)\\
   U^{(l)}\big|_{t=0} = u^{(l)}_f\quad\mbox{in }\mathbb{R}^n
   \end{cases}\qquad l=1,2
   \end{equation}
   with satisfying 
   \begin{equation}\label{Uesti} \|U^{(l)}(\cdot,t)\|_{H^a(\mathbb{R}^n)}\leq C\|u^{(l)}_f\|_{H^a(\mathbb{R}^n)},\end{equation}
   where $C>0$ independent of $t$ and $u^{(l)}_f$ (see \cite{Evans}).
    
   Then one has the following pointwise definition of the  fractional operator $(\mathcal{L}^{(l)})^a$, $\,0<a<1$ acting on $u^{(l)}_f$ as:  
   \begin{equation}\label{gm14}
   \forall x\in \mathbb{R}^n,\quad(\mathcal{L}^{(l)})^a u^{(l)}_f(x) := \frac{1}{\Gamma(-a)}\int_0^\infty \frac{U^{(l)}(x,t)-u^{(l)}_f(x)}{t^{1+a}}\, dt, \quad l=1,2.
   \end{equation}
   
\noindent   
Since from the nonlocal equation \eqref{fgk}  $u^{(1)}_f= u^{(2)}_f$ in $W\subset\Omega_e$, therefore \eqref{L12} and \eqref{gm14} imply that 
   \begin{equation}\label{iid}
   \forall x\in \widetilde{W},\quad \int_0^\infty \frac{U^{(1)}(x,t)-U^{(2)}(x,t)}{t^{1+a}}\, dt = 0.    
   \end{equation}
  Following that, we claim:
   \begin{Proposition}\label{propE}
   	Let $\Theta\Subset W\subset\Omega_e$ and $\Sigma\Subset \widetilde{W}\subset\Omega_e$  be some  non-empty bounded open subsets such that $\overline{\Sigma}\cap \overline{\Theta} =\emptyset$. Let $f\in C^\infty_c(\Theta)$. For any fixed pair of $U^{(1)}, U^{(2)}$ given by \eqref{kh}, the integral identity \eqref{iid} implies that $U^{(1)}=U^{(2)}$ in $[0,\infty)\times\Omega_e$.
   \end{Proposition}
   \begin{proof}
Let us call $U=U^{(1)}-U^{(2)}$. Note that, $supp\, u^{(l)}_f \subseteq \overline{\Omega}\cup \overline{\Theta}$,\, $l=1,2$, and $\big(\overline{\Omega}\cup \overline{\Theta}\big) \cap \overline{\Sigma}=\emptyset$. Therefore while restricting $U(\cdot,t)|_{\Sigma}$ we find
\begin{align}
U(\cdot, t)\big|_{\Sigma}&=U^{(1)}(\cdot,t) \big|_{\Sigma} - U^{(2)}(\cdot,t)\big|_{\Sigma}\notag\\[4pt]
&= \int_{\Omega\cup \Theta}\,\,p^{(1)}_t(x,y)\,u^{(1)}_f(y)\,dy-   \int_{\Omega\cup \Theta}\,\,p^{(2)}_t(x,y)\,u^{(2)}_f(y)\,dy,\quad x\in \Sigma
\end{align}
where $p_t^{(l)}(x,y)$ satisfy the estimate (see \eqref{eq:pointwise estimates for p_t})
\begin{equation}\label{klm3}
C_1\, \Big(\frac{1}{4\pi\,t}\Big)^{\frac{n}{2}} e^{-\frac{\alpha_1\,|x-y|^2}{4t}}\leq p_t^{(l)}(x,y) \leq C_2\,\Big(\frac{1}{4\pi\,t}\Big)^{\frac{n}{2}} e^{-\frac{\alpha_2\,|x-y|^2}{4t}}, \quad x, y\in\mathbb{R}^n
\end{equation}
for some $\alpha_1,\alpha_2, C_1,C_2>0$.

For $x\in \Sigma$ fixed, let us  note that, $\frac{U(x,t)}{t^{m+a}}\in L^1(0,\infty)$ for all $m\in\mathbb{N}$. It follows since: let $0<\delta\ll 1$, and writing 
$\int_0^\infty  |\frac{U(x,t)}{t^{m+a}}|\, dt = \int_0^\delta |\frac{U(x,t)}{t^{m+a}}|\, dt + \int_\delta^\infty |\frac{U(x,t)}{t^{m+a}}|\,dt$, we find that  the second integral is finite as $U(x,\cdot)\in L^\infty(0,\infty)$, and the first integral is finite since $U(x,t)\sim O(t^m)$ for all $m\in\mathbb{N}$  near $t=0$ due to \eqref{klm3} as $\underset{y\in \Omega\cup \Theta}{inf}|x-y|>0$ for $x\in \Sigma$.   

Next, we claim that  
\begin{equation}\label{gm91} \int_0^\infty \frac{U(x,t)}{t^{m+a}}\, dt=0,\quad x\in \Sigma,\quad m\in\mathbb{N}. \end{equation}

Since $\mathcal{L}^{(1)}\big|_{\Omega_e}=\mathcal{L}^{(2)}\big|_{\Omega_e}=(-\Delta)$, so $U$ solves 
   	\begin{equation}\begin{cases}\label{Uh}
   	\partial_t U =\Delta U \mbox{ in }\Omega_e\times (0,\infty)\\[4pt]
   	U(\cdot,0)=0\mbox{ in }\Omega_e
   	\end{cases}\end{equation}
   	 and satisfying 
   	\begin{equation}\label{gm86}\int_0^\infty \frac{U(x,t)}{t^{1+a}}\, dt=0\mbox{ in }\Sigma.\end{equation}

   	Note that, for any $\Sigma\Subset W$, we have $U\in C^\infty(\Sigma\times (0,\infty))$, see \cite{Davies}.
   	So for any $d\in\mathbb{N}$,   $(-\Delta)^dU(\cdot,\cdot)$ solves the heat equation
 \begin{equation}\label{ml2}\begin{cases}
 (\partial_t-\Delta)(-\Delta)^dU(x,t)=0 \mbox{ in } \Sigma\times(0,\infty)\\[4pt]
 (-\Delta)^dU(x,0)=0 \mbox{ on }\Sigma.
 \end{cases}\quad d\in \{0\}\cup \mathbb{N}\end{equation}
 For $x\in \Sigma$ fixed, as the solution of the heat equation with the zero initial data there, it possesses  the fact 
 \[
 \frac{(-\Delta)^d U(x,t)}{t^{1+a}}\in L^1(0,\infty), \quad d\in\{0\}\cup\mathbb{N}.
 \] 
      	    	 Let $d=m+1$, $m\in\mathbb{N}$, and taking $(-\Delta)^{m+1}$  on \eqref{gm86}, 
   	 we get
   	 \[ \int_0^\infty \frac{(-\Delta)^{m+1}U(x,t)}{t^{1+a}}\, dt=0,\quad x\in \Sigma,\quad m\in\mathbb{N} \]
   	 or using the equation \eqref{ml2}:
   	 \[ \int_0^\infty \frac{\partial_t\left((-\Delta)^mU(x,t)\right)}{t^{1+a}}\, dt=0,\quad x\in \Sigma, \quad m\in\mathbb{N} \]
   	 or by doing integration by-parts
   	 \begin{equation}\label{ml1} 
   	 \int_0^\infty \frac{(-\Delta)^mU(x,t)}{t^{2+a}}\, dt=0,\quad x\in \Sigma, \quad m\in\mathbb{N}. \end{equation}
   	Consequently, re-arguing with $d=(m-1),\cdots,(m-k),\cdots,1$ and so on, from \eqref{ml1} one obtains
   	\[\int_0^\infty \frac{U(x,t)}{t^{m+a}}\, dt=0,\quad x\in \Sigma, \quad m\in\mathbb{N}.\]

   	Now for any $\eta\in\mathbb{R}$,  since  $\int_0^\infty \frac{U(x,t)}{t^{1+a}}e^{\frac{i\eta}{t}}\, dt$ exists as  $ \frac{U(x,t)}{t^{1+a}}\in L^1(0,\infty)$ for $x\in \Sigma$ fixed, so \eqref{gm91} implies   
   	\[\forall \eta\in\mathbb{R},\quad \int_0^\infty \frac{U(x,t)}{t^{1+a}}\,e^{\frac{i\eta}{t}}\, dt=0,\quad x\in \Sigma. \] 
   	Therefore realizing the above integral as one-dimensional Fourier transform:  $\widehat{V}_x(\eta)=\int_\mathbb{R} V_x(\lambda) \, e^{i\eta\lambda}\,d\lambda=0$, where $V_x(\lambda)=\chi_{(0,\infty)}(\lambda)\frac{U(x,\lambda^{-1})}{\lambda^{1-a}}$ we conclude $U(x,\cdot) =0$ in $(0,\infty)$, $x\in \Sigma$. In particular, $U=0$ in $\Sigma\times (0,\infty)$. 
   	Then by the unique continuation of the infinite propagation of heat we conclude that $U=0$ everywhere in $\Omega_e\times (0,\infty)$ as the solution of \eqref{Uh}. We refer \cite{Miller} for the unique continuation result. 
   	This completes the proof. 
   	\hfill\end{proof}
   As an application of the above result we will deriving certain unique continuation principal for the operator $\mathcal{L}^a$, $0<a<1$.  
   \begin{Proposition}[Unique continuation principle]\label{propU}
   	Let $\Omega$ be a bounded domain in $\R^n$ and $\Sigma\subset \Omega_e$ be any non-empty open set.  Let the operator 
   	 $\mathcal{L}$ introduced in  \eqref{lopera} satisfy the conditions mentioned in \eqref{eq:ellipticity and symmetry condition}-\eqref{eq:condi_bc}, and further we assume $a_{jk}\big|_{\Sigma}=\delta_{jk}$. Let for some $u\in H^a(\R^n)$,  $u|_{\Sigma}=\mathcal{L}^au|_{\Sigma}=0$. Then $u\equiv 0$.   
   \end{Proposition}
   \begin{Remark}
   	This result has been already proved in \cite{GSU, GLX} for the fractional operator $\mathcal{L}_0^a$ where $\mathcal{L}_0$ denotes the principal part of $\mathcal{L}$ in \eqref{lopera}. 
   	
   	However, the method of proof presented here is different from the method presented in  \cite{GSU, GLX, GRSU} and does not require any regularity assumption on $A$ as it does in \cite{GLX}.
   	\end{Remark}
   	\begin{proof}[Proof of Proposition \ref{propU}]
   	The proof follows from the proof of the above Proposition \ref{propE}.  Let $u\in H^a(\mathbb{R}^n)$
   	and $U= \int_{\R^n}p_t(x,y)u(y)\,dy\, \in C\left([0, \infty); H^a(\R^n)\right)$ be the unique solution of
   	\[
   	\begin{cases}
   	\partial_t U = \mathcal{L}\, U \quad\mbox{in }\mathbb{R}^n\times (0,\infty)\\
   	U\big|_{t=0} = u\quad\mbox{in }\mathbb{R}^n.
   	\end{cases}
   	\]
   	Now from the hypotheses of our Proposition as it follows: $U$ satisfies 
   	\[
   	\partial_t U =\Delta U \mbox{ in }\Sigma\times (0,\infty)
   	\]
   	with $ U(\cdot,0)=0$ in $\Sigma$ and  
   	\[
   	\int_0^\infty \frac{U(x,t)}{t^{1+a}}\, dt=0\mbox{ in }\Sigma.
   	\] 
  Thus by following the proof of Proposition \ref{propE}, we get $U=0$ in $\Sigma\times (0,\infty)$, and then by unique continuation (see \cite{VESE}) of the solution of the parabolic equation it follows $U \equiv 0$ implying $u \equiv 0$ as well. This completes the proof. 
  \hfill
   \end{proof}

Continuing the proof of the Theorem \ref{thm1}:   
   Let us define the function 
   \begin{equation}\label{Phi}\Phi^{(l)}(x)= \int_0^\infty {U}^{(l)}(x,t)\, dt, \quad l=1,2\end{equation}
   and 
   thanks to the above Proposition \ref{propE}, we have 
   \[
   \Phi^{(1)} = \Phi^{(2)} \quad\mbox{in }\Omega_e.
   \]
   Note that, by definition \eqref{Phi}, $\Phi^{(l)}\in L^2(\mathbb{R}^n)$ and it solves 
   \begin{equation}\label{ellip}
   \mathcal{L}^{(l)}\Phi^{(l)} = u^{(l)}_f \quad\mbox{in }\R^n, \quad l=1,2
   \end{equation}
   in the sense of distribution. 
   
   Then further using the regularity result (cf.\cite[Theorem 6.12]{GRUBB}) we  conclude that $\Phi^{(l)}\in H^{a+2}(\R^n)$, as $u^{(l)}_f\in H^a(\R^n)$.
   
   Let us call, 
   \begin{equation}\label{psik} \Psi^{(l)} = (\mathcal{L}^{(l)})^a\Phi^{(l)} \quad\mbox{in }\R^n. \end{equation}
   We find that $\Psi^{(l)}\in H^{2-a}(\R^n)$, as $\Phi^{(l)}\in H^{2+a}(\R^n)$.
   
   Since
   \[(\mathcal{L}^{(l)})\left((\mathcal{L}^{(l)})^a\Phi^{(l)}\right) =(\mathcal{L}^{(l)})^a\left((\mathcal{L}^{(l)})\Phi^{(l)}\right)\mbox{  in }\R^n\]
   so from \eqref{ellip} it follows that 
   \begin{equation}\label{ellipN}
   \mathcal{L}^{(l)}\Psi^{(l)} =(\mathcal{L}^{(l)})^a\,u^{(l)}_f \quad\mbox{in }\mathbb{R}^n 
   \end{equation}
   and in particular from \eqref{fgk}, we have
   \begin{equation}\label{ellip2}
  \mathcal{L}^{(l)}\Psi^{(l)} =0 \quad\mbox{in }\Omega,\quad l=1,2.
   \end{equation}
   We would like to show 
   \begin{equation}\label{ps12}
   \Psi^{(1)} = \Psi^{(2)} \quad\mbox{in }\Omega_e.
   \end{equation}
   Note that, since $\Psi^{(l)}\in H^{2-a}(\mathbb{R}^n)$, and $0<a<1$, so $\Psi^{(l)}\in H^1(\mathbb{R}^n)$ for $l=1,2$. 
   
   Now let us consider 
   \[V^{(l)}(x,t) = \int_{\R^n}p^{(l)}_t(x,y)\Phi^{(l)}(y)\,dy,\qquad l=1,2\] solving 
   \begin{equation}\label{kh2}
   \begin{cases}
   \partial_t V^{(l)} =\mathcal{L}^{(l)}V^{(l)} \quad\mbox{ in }\R^n\times (0,\infty)\\
   V^{(l)}(\cdot,0)=\Phi^{(l)} \quad\mbox{ in }\R^n
   \end{cases}\qquad l=1,2.
   \end{equation}
   We note that, $V^{(l)}$ defined above belongs to $C\left([0, \infty); H^{a+2}(\R^n)\right)\,\cap\, C^\infty\left(\R^n\times (0,\infty)\right)$ be the unique solution of \eqref{kh2}.

   Then by taking the action of $\mathcal{L}^{(l)}$ over the equation \eqref{kh2} with respect to the space variables,  $\widetilde{V}^{(l)}(x,t)=\mathcal{L}^{(l)}V^{(l)}(x,t)$ solves 
   \begin{equation}\label{kh3}
   \begin{cases}
   \partial_t \widetilde{V}^{(l)} =\mathcal{L}^{(l)}\widetilde{V}^{(l)} \mbox{ in }\R^n\times (0,\infty)\\
   \widetilde{V}^{(l)}(\cdot,0)=u^{(l)}_f \mbox{ in }\R^n
   \end{cases}\qquad l=1,2
   \end{equation}
   due to \eqref{ellip}. 
   
   By the definition of $V^{(l)}$, $\widetilde{V}^{(l)}$, and using the symmetric action of the heat kernel $\langle \mathcal{L}^{(l)} p_t(x,y), \varphi(y)\rangle_y= \langle p_t(x,y), \mathcal{L}^{(l)}\varphi(y)\rangle_y $ for $\varphi\in C^\infty_c(\R^n)$, it follows that $\widetilde{V}^{(l)}(x,t)= \int_{\R^n} p_t(x,y)\, \mathcal{L}^{(l)}\Phi^{(l)}(y)\,dy$ 
   belongs to $C\left([0, \infty); H^{a}(\R^n)\right)$ as the unique solution of \eqref{kh3}, with satisfying the estimate $\|\widetilde{V}^{(l)}(\cdot,t)\|_{H^a(\mathbb{R}^n)}\leq C\|u^{(l)}_f\|_{H^a(\mathbb{R}^n)}$, 
where $C>0$ independent of $t$ and $u^{(l)}_f$ (see \eqref{Uesti}).
   
   Consequently, by the uniqueness of the solution of the heat equations \eqref{kh} and \eqref{kh3}, we conclude 
   \begin{equation}\label{uv}
   U^{(l)} = \widetilde{V}^{(l)}=\mathcal{L}^{(l)}V^{(l)} \quad\mbox{in }\R^n\times (0,\infty).
   \end{equation}
   
   Therefore, from \eqref{kh2}, we conclude
   \begin{align*}\label{ide}
   \partial_t V^{(1)} - \partial_t V^{(2)}  = \mathcal{L}^{(1)}V^{(1)} - \mathcal{L}^{(2)}V^{(2)} &= U^{(1)}-U^{(2)}\\[1mm]\notag
   &=0 \mbox{ in }\Omega_e\times (0,\infty)
   \end{align*}
   thanks to the Proposition \ref{propE}.
   
   Hence, \[(V^{(1)}-V^{(2)})(\cdot,t)=(V^{(1)} -V^{(2)} )(\cdot,0) \quad\mbox{in }\Omega_e\times (0,\infty)\]
   or,
   \[V^{(1)} (\cdot,t) -V^{(1)} (\cdot,0) = V^{(2)} (\cdot,t) -V^{(2)} (\cdot,0) \quad\mbox{in }\Omega_e\times (0,\infty)\]
   or,
   \[
   \forall x\in \Omega_e,\quad \int_0^\infty \frac{V^{(1)} (x,t)-V^{(1)} (x,0)}{t^{1+a}}\, dt = \int_0^\infty \frac{V^{(2)} (x,t)-V^{(2)} (x,0)}{t^{1+a}}\, dt.
   \]
   This shows, by applying the definition \eqref{gm14} on \eqref{kh2}:
   \[(\mathcal{L}^{(1)})^a\,\Phi^{(1)}  = (\mathcal{L}^{(2)})^a\,\Phi^{(2)} \quad\mbox{in }\Omega_e.\]
   Hence we have shown \eqref{ps12}, i.e. $\Psi^{(1)} =\Psi^{(2)} $ in $\Omega_e$. 
   
   Consequently,  \eqref{ellip2}-\eqref{ps12} imply that  
   \begin{equation}\label{local}
   \begin{cases}
   \mathcal{L}^{(l)}\,\Psi^{(l)}  =0 \quad\mbox{in }\Omega\\
   (\Psi^{(1)} , \partial_{\nu^{(1)}} \Psi^{(1)} )= (\Psi^{(2)} , \partial_{\nu^{(2)}}\Psi^{(2)} ) \mbox{ on }\partial\Omega.
   \end{cases}
   \end{equation}
   Recall that, $\Psi^{(l)} \in H^{2-a}(\mathbb{R}^n)$, $0<a<1$. Therefore $\Psi^{(l)} $ is in $H^1(\Omega)$, and on the boundary $\partial\Omega$ the boundary Cauchy data $(\Psi^{(l)} , \partial_{\nu^{(l)}} \Psi^{(l)} )$ is well defined in $H^{\frac{1}{2}}(\partial\Omega)\times H^{-\frac{1}{2}}(\partial\Omega)$ for $l=1,2$.  
   
   Next, we would like to vary all possible $f$ in \eqref{fgk} to see whether it allows the all possible Cauchy data in \eqref{local} to address the local inverse problem of determining the coefficients $a^{(1)}_{jk}=a^{(2)}_{jk}$, $b^{(1)}_j=b^{(2)}_j$ and $c^{(1)}=c^{(2)}$ uniquely in $\Omega$. 
   
   \subsection{Local inverse problem:}\label{lip}
   Based on the development what we have made so far, here we will show that 
   the boundary Cauchy data  $\mathcal{C}^{(l)}_{\partial\Omega}=\{\mathscr{V}^{(l)}|_{\partial\Omega}, \partial_{\nu^{(l)}} \mathscr{V}^{(l)}|_{\partial\Omega}\}$
   for all the solutions $\{\mathscr{V}^{(l)}\in H^1(\Omega): \mathcal{L}^{(l)}\mathscr{V}^{(l)}=0 \mbox{ in } \Omega\}$  $l=1,2$ 
    are equal, i.e. 
    \begin{equation}\label{lchy}
    \mathcal{C}^{(1)}_{\partial\Omega}=\mathcal{C}^{(2)}_{\partial\Omega}.
    \end{equation} 
   
   Let us define two spaces
   \[
   \mathcal{S}^{(l)}(\Omega) = \{\mathscr{V}^{(l)}\in H^1(\Omega):\, \mathcal{L}^{(l)}\,\mathscr{V}^{(l)}=0 \mbox{ in } \Omega\},\,\, l=1,2
   \]
   and 
   \[
   \widetilde{\mathcal{S}}^{(l)}(\Omega)=\{\Psi^{(l)} |_{\Omega}:\,\Psi^{(l)} \mbox{ solves \eqref{ellipN} in $\mathbb{R}^n$}\},\,\, l=1,2.
   \]
   \begin{Lemma}\label{lemL}
   	$\widetilde{\mathcal{S}}^{(l)}(\Omega)$ is dense in $\mathcal{S}^{(l)}(\Omega)$ in $H^1$-norm topology. 
   \end{Lemma}
   
   Suppose the above Lemma is true. It says that for a given any $\mathscr{V}^{(1)}\in \mathcal{S}^{(1)}(\Omega)$, and $\mathscr{V}^{(2)}\in \mathcal{S}^{(2)}(\Omega)$ with $\mathscr{V}^{(1)}=\mathscr{V}^{(2)}=h$ on $\partial\Omega$, where $h\in H^{\frac{1}{2}}(\partial\Omega)$ arbitrary,  there exist sequences  $\{\Psi^{(l)}_k\}_{k\in\mathbb{N}}$ solving \eqref{ellipN}, such that $\{\Psi^{(l)}_k\}\big|_{\Omega} \to \mathscr{V}^{(l)}$ strongly in $H^1(\Omega)$ as $k\to \infty$. Since $\{\Psi^{(l)}_k\}\big|_{\Omega}$ in $H^1(\Omega)$ solves the local system \eqref{local}, that
   \[\begin{cases}
   \mathcal{L}^{(l)}\Psi^{(l)}_k=0 \quad\mbox{in }\Omega,\\
   (\Psi^{(1)}_k, \partial_{\nu^{(1)}} \Psi^{(1)}_k)= (\Psi^{(2)}_k, \partial_{\nu^{(2)}}\Psi^{(2)}_k) \quad\mbox{on }\partial\Omega.
   \end{cases}\quad  k\in \mathbb{N}
   \]
   Therefore, we obtain 
   \begin{equation}\label{lan}
   \begin{cases}
   \mathcal{L}^{(l)}\mathscr{V}^{(l)}=0 \quad\mbox{in }\Omega,\\
   (\mathscr{V}^{(1)}, \partial_{\nu^{(1)}} \mathscr{V}^{(1)})= (\mathscr{V}^{(2)}, \partial_{\nu^{(2)}}\mathscr{V}^{(2)}) \quad\mbox{on }\partial\Omega
   \end{cases} 
   \end{equation}
   i.e., in other words \eqref{lchy} follows.
   
   Thus we have reduced our nonlocal inverse problem (cf. Subsection \ref{nlip}) into solving a local inverse problem of determining the coefficients $a^{(1)}_{jk}=a^{(2)}_{jk}$, $b^{(1)}_j=b^{(2)}_j$ and $c^{(1)}=c^{(2)}$ uniquely in $\Omega$ 
   from the equality of the boundary Cauchy data  $\mathcal{C}^{(1)}_{\partial\Omega}=\mathcal{C}^{(2)}_{\partial\Omega}$.

   Now it remains to prove Lemma \ref{lemL}. We do it in two steps. Let us begin with this following density result. 
   \begin{Proposition}\label{dL5}
   		Let $\Omega$ and $W$ be two bounded non-empty open set in $\mathbb{R}^n$ such that $\overline{\Omega}\cap \overline{W}=\emptyset$. 	Let $0<a<1$ and $u\in H^a(\R^n)$ solves 
   	\begin{equation}\label{dL1}
   	\mathcal{L}^au=0 \mbox{ in }\Omega,\quad supp\, u\subseteq \overline{\Omega}\cup\overline{W}.
   	\end{equation}
   	Then for any open set $E\subseteq \mathbb{R}^n\setminus \overline{(\Omega\cup W)}$, the set 
   	\[ \mathcal{N}(E):=\{\mathcal{L}^au\,\big|_{E}:\,\, u\mbox{ solves }\eqref{dL1}\} \]
   	is dense in $H^{-a}(E)$. 
   \end{Proposition}
   
   \begin{Remark}\label{dL8}
   	We remark here that by varying $u|_{W}$ in $C^\infty_c(W)$ where $u$ solves \eqref{dL1}, while the set  $\{\mathcal{L}^au|_{E}\}$ remains bounded in $H^{-a}(E)$, however the set $\{\mathcal{L}^au|_{W}\}$ do not necessarily remained bounded in $H^{-a}(W)$. See \cite{RS1, RS4} in this direction. 
   \end{Remark}
   
   \begin{proof}[Proof of Proposition \ref{dL5}]
   	
   	In order to prove the required density result, by using the Hahn-Banach theorem, it is enough to show that, if for some $h\in H^a_0(E)$ 
   	\begin{equation}\label{dL10}
   	\langle \mathcal{L}^au, h\rangle_{(H^{-a}(E), H^a_0(E))}=0\quad\mbox{ for all $u$ solving \eqref{dL1}}
   	\end{equation}
   	then it must follow $h\equiv 0$. 
   	
   	Let us consider the adjoint problem, that $v\in H^a(\mathbb{R}^n)$ solving  
   	\begin{equation}\label{dL2}
   	\begin{cases}\mathcal{L}^av=0 \quad\mbox{ in }\Omega\\
   	v= h \quad\mbox{ in }E\\
   	v=0 \quad\mbox{ in }\mathbb{R}^n\setminus \overline{(\Omega\cup E)}
   	\end{cases}\end{equation}	
   	
   	Hence we find from \eqref{dL1}, \eqref{dL2}, and \eqref{dL10} that 
   	\begin{align*} 
   	\langle u, \mathcal{L}^av\rangle_{W} 
   	&= \langle u, \mathcal{L}^av\rangle_{\mathbb{R}^n} 
   	-\langle u, \mathcal{L}^av\rangle_{\Omega}\\
   	&=\langle \mathcal{L}^au, v\rangle_{\mathbb{R}^n}\\
   	&= \langle \mathcal{L}^au, h\rangle_{E}=0.   
   	\end{align*} 
   	So by varying $u|_{W}\in C^\infty_c(W)$, we obtain $\mathcal{L}^av=0$ in $W$. Since $W\subset \mathbb{R}^n\setminus \overline{(\Omega\cup E)}$, and $v=0$ there in $W$. Thus it follows from Proposition \ref{propU} that $v\equiv 0$. Consequently, it implies  $h= 0$. This completes the proof. 
   	\hfill\end{proof}
   
   Now we complete the proof of Lemma \ref{lemL}. 
   
   \begin{proof}[Proof of Lemma \ref{lemL}]
   	We want to show the space $\widetilde{\mathcal{S}}^{(l)}(\Omega)$ is dense in $S^{(l)}(\Omega)$ in $H^1(\Omega)$ strong topology.  As usual, we invoke the Hahn-Banach theorem to prove our result, by saying if for some $F\in \widetilde{H}^{-1}(\Omega)$ $(=(H^1(\Omega))^{*})$, 
   	\begin{equation}\label{dL7} 
   	\langle F, \Psi^{(l)} \rangle_{ \widetilde{H}^{-1}(\Omega), H^{1}(\Omega)} = 0 \quad\mbox{for all }\Psi^{(l)} \in \widetilde{\mathcal{S}}^{(l)}(\Omega) 
   	\end{equation}
   	then it must follow
   	\begin{equation}\label{dL6}
   	 \langle F, \mathscr{V}^{(l)}\rangle_{ \widetilde{H}^{-1}(\Omega), H^{1}(\Omega)} = 0 \quad\mbox{for all }\mathscr{V}^{(l)}\in \mathcal{S}^{(l)}(\Omega) 
   	 \end{equation}
   	for corresponding $l=1, 2$ respectively.

   	We recall that, $\widetilde{H}^{-1}(\Omega)$ be the dual space of $H^1(\Omega)$ defined as
   	\[  \widetilde{H}^{-1}(\Omega):=\{ F\in H^{-1}(\mathbb{R}^n):\,\, supp\, F\subseteq \overline{\Omega} \}  \] 	
   	with the duality bracket
   	\[ \langle  F, \phi\rangle_{\widetilde{H}^{-1}(\Omega), H^{1}(\Omega)} = \langle  F, \widetilde{\phi}\rangle_{H^{-1}(\mathbb{R}^n), H^{1}(\mathbb{R}^n)}     \]	
   	where $\phi\in H^1(\Omega)$ and $\widetilde{\phi}$  be its any  $H^1(\mathbb{R}^n)$ extension (i.e. $\widetilde{\phi}\in H^1(\mathbb{R}^n)$ and $\widetilde{\phi}|_{\Omega} =\phi$).\\
   	\\
   	Let $\Omega$, $W$ are as in Proposition \ref{dL5}, and denote $E:=\mathbb{R}^n\setminus \overline{(\Omega\cup W)}$. Let $0<a<1$ and we recall \eqref{ellipN} as $\Psi^{(l)} \in H^{1}(\mathbb{R}^n)$ solving
   	\begin{equation}\label{dL9}
   	-\mathcal{L}^{(l)}\Psi^{(l)}  =(\mathcal{L}^{(l)})^au^{(l)}_f \quad\mbox{in }\mathbb{R}^n.  
   	\end{equation}
   	Now by varying $f\in C^\infty_c(W)$, the Proposition \ref{dL5} gives us the space 
   	\begin{equation}\label{st1}
    \mathcal{N}_k(E)=\{\mathcal{L}^{(l)}\Psi^{(l)} \,|_{E}\,:\,\Psi^{(l)}  \mbox{ solves \eqref{dL9} is dense in }H^{-1}(E). 
\end{equation}   
\vspace{3pt}
\noindent
Note that, as pointed out as in Remark \ref{dL8},  we have proven the density result (cf. Proposition \ref{dL5}) in $E$ only, and not necessarily in $\mathbb{R}^n\setminus \overline{\Omega}$. 
   	\vspace{3pt}
   	
   	Following that, let us consider $\widetilde{\Psi}^{(l)}\in H^1(\mathbb{R}^n)$ as 
   	\[\widetilde{\Psi}^{(l)}= \begin{cases} \Psi^{(l)}  \mbox{ in }\mathbb{R}^n\setminus \overline{W}\\
   	\mathscr{U}^{(l)} \mbox{ in }W
   	\end{cases}\]
   	where, $\mathscr{U}^{(l)}\in H^1(W)$ is defined as 
   	\[\begin{cases} (-\mathcal{L}^{(l)})\mathscr{U}^{(l)} = 0 \quad\mbox{ in }W\\
   	\mathscr{U}^{(l)} = \Psi^{(l)}  \quad\mbox{ on }\partial W.
   	\end{cases}
   	\]
   	Clearly, $\widetilde{\Psi}^{(l)}\in H^1(\mathbb{R}^n)$ and $\mathcal{L}^{(l)}\widetilde{\Psi}^{(l)} \in \widetilde{H}^{-1}(\mathbb{R}^n\setminus \overline{W})$ as
   	\begin{equation}\label{dL4}
   	\mathcal{L}^{(l)}\widetilde{\Psi}^{(l)}=\begin{cases}\mathcal{L}^{(l)}{\Psi}^{(l)}&\quad\mbox{in }\mathbb{R}^n\setminus \overline{W}\\
   	0&\quad\mbox{in }W
   	\end{cases}.
   	\end{equation}
   	
   	Now let us assume for some $F\in \widetilde{H}^{-1}(\Omega)$, \eqref{dL7} holds. Then from there we write
   	\begin{align} 0 &=\langle F, \Psi^{(l)} \rangle_{ \widetilde{H}^{-1}(\Omega), H^{1}(\Omega)}\notag\\
   	&=\langle  F, \widetilde{\Psi}^{(l)}\rangle_{ {H}^{-1}(\mathbb{R}^n), H^{1}(\mathbb{R}^n)}.\label{dL11}
   	\end{align}
   	Since, $F\in H^{-1}(\mathbb{R}^n)$ with $supp\, F\subseteq\overline{\Omega}$;  So there exists a $\varPhi^{(l)}\in H^1(\mathbb{R}^n)$ uniquely solving (cf.\cite[Theorem 6.12]{GRUBB})
   	\begin{equation}\label{dL3} \mathcal{L}^{(l)}\varPhi^{(l)} = F \quad\mbox{ in }\mathbb{R}^n.
   	\end{equation}
   	Then from \eqref{dL11}, we simply obtain 
   	\[  
   	\langle \mathcal{L}^{(l)}\widetilde{\Psi}^{(l)}, \varPhi^{(l)}\rangle_{ {H}^{-1}(\mathbb{R}^n), H^{1}(\mathbb{R}^n)}=0.
   	\]
   	Next, from \eqref{dL4}, together with using the fact $\mathcal{L}^{(l)}\Psi^{(l)} =0$ in $\Omega$,  we find
   	\[  
   	\langle \mathcal{L}^{(l)}\Psi^{(l)} , \varPhi^{(l)}\rangle_{ {H}^{-1}(E), H^{1}(E)}=0 \quad\mbox{ for all } \Psi^{(l)}  \mbox{ solving }\eqref{dL9}.
   	\]
   	\vspace{3pt} 
   	Since all possible $\{\mathcal{L}^{(l)}\Psi^{(l)} \,|_{E}\,:\,\Psi^{(l)}  \mbox{ solves }\eqref{dL9}\}$ is dense in $H^{-1}(E)$ (see \eqref{st1}). Therefore $\varPhi^{(l)}=0$ in $E$.  
   	
   	Since $\mathcal{L}^{(l)}\varPhi^{(l)}=0$ in $\mathbb{R}^n\setminus\overline{\Omega}$ (cf. \eqref{dL3}), so  $\varPhi^{(l)}=0$ in $E$ implies  $\varPhi^{(l)}=0$ in $\mathbb{R}^n\setminus\overline{\Omega}$,  thanks to the unique continuation property of the elliptic partial differential operator, see \cite{Wolff}.
   	
   	Therefore, we find $\varPhi^{(l)}\in H^1_0(\Omega)$ which solves  
   	$\mathcal{L}^{(l)}\varPhi^{(l)} = F\in \widetilde{H}^{-1}(\Omega)$, which implies $\partial_{\nu^{(l)}}\varPhi^{(l)}\,|_{\partial\Omega}$ is well-defined in $H^{-\frac{1}{2}}(\partial\Omega)$, and from above  $\varPhi^{(l)}=0$ in $\mathbb{R}^n\setminus\overline{\Omega}$,   we actually have $\partial_{\nu^{(l)}}\varPhi^{(l)}\,|_{\partial\Omega}=0$.  
   	
   	Thus \eqref{dL6} follows, as integration by-parts gives
   	\begin{align*}\langle F, \mathscr{V}^{(l)}\rangle_{ \widetilde{H}^{-1}(\Omega), H^{1}(\Omega)} &=  \langle  \mathcal{L}^{(l)}\mathscr{V}^{(l)}, \varPhi^{(l)}\rangle_{ \widetilde{H}^{-1}(\Omega), H^{1}(\Omega)}\\
   	&=0 \quad\mbox{ for all $\mathscr{V}^{(l)}\in \mathcal{S}^{(l)}(\Omega)$}.
   	\end{align*}
   	This completes the proof.  
   	\hfill\end{proof}
   
   Hence by reducing  our nonlocal inverse problem (cf. Subsection \ref{nlip}) into a local inverse problem (cf. Subsection \ref{lip}) we have completed the proof of Theorem \ref{thm1}.

\section{Gauge invariance in the nonlocal case}\label{sec4}

We address the following question here, does the nonlocal inverse problem (cf. (A)) exhibit the gauge invariance?   Let us discuss about it here.
		\subsection{Transformation optics}\label{top}
		
		In general, the idea of transformation optics or the invariance under the transformation,  has had appeared in the literature on studying ``non-uniqueness of Calder\'{o}n problem" and ``cloaking", see the papers \cite{GLU2,GLU, P}.	We refer to the survey paper \cite{GKLU} and the references there in for more details.
		
		Now we will be discussing these related issues in nonlocal settings.   
		
		\subsection*{Non-local case:}
		Let us talk about the ``transformation optics" or the change of variables techniques in particular for the heat equation, which have been used to  define the fractional operators.

		Let us consider a locally Lipschitz invertible map $\mathbb{F}:\R^n\mapsto\R^n$ such that $\mathbb{F}(x) = x$ for each $x\in \R^n\setminus B_\rho$, where $B_\rho=B(0;\rho)$ denotes an euclidean ball of radius $\rho$ centered at the origin.  Furthermore, assume that the associated Jacobians satisfy 
		\begin{equation}\label{jaco}
		{\rm det}(D\mathbb{F})(x),\, {\rm det}(D\mathbb{F}^{-1})(x) \ge C >0 \mbox{ for a.e. }x\in\R^n.
		\end{equation} 
		Then one has the following proposition known as transformation optics. 
		\begin{Proposition}\label{cov}
			$U$ is a solution to 
			\begin{equation}\label{U}
			\partial_t U = \nabla \cdot \Big( A(x) \nabla U \Big), \quad (x,t)\in \mathbb{R}^n\times (0,\infty)
			\end{equation}
			if and only if $V= U\circ \mathbb{F}^{-1}$ is a solution to
			\begin{equation}\label{V}
			\mathbb{F}_{*}1(y) \partial_t V = \nabla \cdot \Big( \mathbb{F}_{*}A(y) \nabla V \Big),  \quad (y,t)\in \mathbb{R}^n\times (0,\infty)
			\end{equation}
			where the coefficients are given as
			\begin{equation}\label{pff}
			\mathbb{F}_{*} 1(y) = \frac{1}{{\rm det }(D\mathbb{F})(x)},
			\quad 
			\mathbb{F}_{*}A(y) = \frac{D\mathbb{F}^{\top}(x) A(x) D\mathbb{F}(x)}{{\rm det }(D\mathbb{F})(x)}
			\end{equation}
			with the understanding that the right hand sides in \eqref{pff} are computed at $x=\mathbb{F}^{-1}(y)$. Moreover we have for all $t>0$,
			\begin{align}\label{eq:prop:change-variable-assertion}
			U(t,\cdot) = V(t,\cdot) \qquad \mbox{ in }\R^n\setminus B_\rho.
			\end{align}
		\end{Proposition}
		The above claim essentially follows from performing a change of variables in the weak formulation associated with the differential equation. In particular for the heat equation it has been presented in \cite[Section 2]{GAV}, \cite{CGHP}.

		Let us seek those $\mathbb{F}$  locally Lipschitz invertible map $\mathbb{F}:\R^n\mapsto\R^n$ satisfying \eqref{jaco},  such that $\mathbb{F}(x) = x$ in $W\subset\mathbb{R}^n$ be some non-empty open set. 
		Let us call $U(\cdot,0)=u(\cdot)$ and $V(\cdot,0)=v(\cdot)$. By definition, $V= U\circ \mathbb{F}^{-1}$ so it gives $u=v$ in $W$.     
		Let us call (cf. \eqref{gm12})
		\[\mathcal{L}_A:=-\sum_{jk=1}^n \frac{\partial}{\partial x_j}a_{jk}(x)\frac{\partial}{\partial x_k}.\]
		Then we have
		\begin{align*}
		\forall x\in W, \quad \big(\mathcal{L}_A\big)^a u(x)&= \frac{1}{\Gamma(-a)}\int_0^\infty \frac{U(x,t)-u(x)}{t^{1+a}}\, dt\\[5pt]
		&= \frac{1}{\Gamma(-a)}\int_0^\infty \frac{V(x,t)-v(x)}{t^{1+a}}\, dt\\[5pt]
		&=  \big(\mathcal{L}_{\mathbb{F}_{*}A}\big)^a v(x).
		\end{align*}
		Hence 
		\begin{equation}\label{ecd}
		\Big(u(x), \big(\mathcal{L}_A\big)^a  u(x)\Big)\Big|_{W}=\Big(v(x), \big(\mathcal{L}_{\mathbb{F}_{*}A}\big)^a v(x)\Big)\Big|_{W}.
		\end{equation} 
		Let $\Omega\subset\mathbb{R}^n$ be some open set such that $\overline{\Omega}\cap \overline{W}=\emptyset$. Let $u$ satisfies $\big(\mathcal{L}_A\big)^a  u(x)=0$ in $\Omega$, then we find $v$ satisfies  
		$\big(\mathcal{L}_{\mathbb{F}_{*}A}\big)^a v(x)=0$ in $\Omega$. Since $\mathbb{F}$ is a diffeomorphism and $U=V\circ \mathbb{F}^{-1}$ thus
		\[
		\forall x\in \Omega, \quad 0=\big(\mathcal{L}_A\big)^a  u(x)= \frac{1}{\Gamma(-a)}\int_0^\infty \frac{U(x,t)-u(x)}{t^{1+a}}\, dt\]
		implies that,
		\[\forall x\in \Omega, \quad \big(\mathcal{L}_{\mathbb{F}_{*}A}\big)^a v(x)=\frac{1}{\Gamma(-a)}\int_0^\infty \frac{V(x,t)-v(x)}{t^{1+a}}\, dt=0.  
		\]
		This shows corresponding to two different matrices $A$ and $\mathbb{F}_{*}A$, the exterior Cauchy data  $\mathcal{C}^A_{(W, W)}=\left(u|_{W}, \big(\mathcal{L}_{A}\big)^a u|_{W}\right)$ and $ \mathcal{C}^{\mathbb{F}_{*}A}_{(W, W)}= \left(v|_{W}, \big(\mathcal{L}_{\mathbb{F}_{*}A}\big)^a v|_{W}\right)$  are same (cf. \eqref{ecd}) for the two sets of solutions 
		$\{u\in H^a(\mathbb{R}^n): \big(\mathcal{L}_{A}\big)^au=0 \mbox{ in } \Omega\}$ and $\{v\in H^a(\mathbb{R}^n): \big(\mathcal{L}_{\mathbb{F}_{*}A}\big)^a v=0 \mbox{ in } \Omega\}$ respectively, $0<a<1$.
		
		 Let us concise it in the theorem below. 	
		
		\begin{Theorem}\label{thm6}
		 Let $\Omega\subset \R^n$ be some bounded non-empty open set, and $W \subset\Omega_e$ be an another non-empty open set such that $\overline{\Omega}\cap\overline{W}=\emptyset$.  Let $\mathbb{F}:\R^n\mapsto\R^n$ be a locally Lipschitz, invertible map satisfying \eqref{jaco}, such that $\mathbb{F}(x) = x$ for each $x\in W$.
			Then for the two different matrices $A$ and $\mathbb{F}_{*}A$ satisfying \eqref{eq:ellipticity and symmetry condition}, the exterior Cauchy data are same i.e. $\mathcal{C}^A_{(W, W)}=\mathcal{C}^{\mathbb{F}_{*}A}_{(W, W)}$  for the two sets of solutions $\{u\in H^a(\R^n): \big(\mathcal{L}_{A}\big)^au=0 \mbox{ in } \Omega\}$ and $\{v\in H^a(\R^n): \big(\mathcal{L}_{\mathbb{F}_{*}A}\big)^a v=0 \mbox{ in } \Omega\}$ respectively, $0<a<1$.
		\end{Theorem}
	
	This settles the question that even in nonlocal case we can not except the full recovery, it always possesses with the change of variable invariance.

		\end{document}